\documentclass[review]{elsarticle}

\usepackage{lineno,hyperref}
\modulolinenumbers[5]

\journal{Journal of \LaTeX\ Templates}

\usepackage{amsthm}
\usepackage[top=2cm, bottom=2cm, left=2cm, right=2cm]{geometry}

\usepackage{xcolor} 

\def\Xint#1{\mathchoice
{\XXint\displaystyle\textstyle{#1}}%
{\XXint\textstyle\scriptstyle{#1}}%
{\XXint\scriptstyle\scriptscriptstyle{#1}}%
{\XXint\scriptscriptstyle\scriptscriptstyle{#1}}%
\!\int}
\def\XXint#1#2#3{{\setbox0=\hbox{$#1{#2#3}{\int}$ }
\vcenter{\hbox{$#2#3$ }}\kern-.6\wd0}}

\def\dashint{\Xint-}

\usepackage{amsfonts}
\usepackage{amsthm}
\usepackage{amssymb}
\usepackage{amsmath}
\usepackage{mhchem}
\usepackage{mathtools}

\newtheorem*{thm*}{Theorem}
\newtheorem{thm}{Theorem}
\newtheorem{lemma}[thm]{Lemma}
\newtheorem{proposition}[thm]{Proposition}
\newtheorem{remark}[thm]{Remark}

\usepackage{amsthm}
\theoremstyle{plain}

\begin{document}

\begin{frontmatter}

\title{Nonlocal reaction preventing blow-up in the supercritical case of chemotaxis}
\author{Evangelos A. Latos\footnote{Institut f\"ur Mathematik und
    Wissenschaftliches Rechnen, Heinrichstra{\ss}e 36, 8010
    Graz, Austria. Email: \texttt{evangelos.latos@uni-graz.at}.}
        }

\begin{abstract}
This paper studies the non-negative solutions of the Keller-Segel model with a nonlocal nonlinear source in a bounded domain. The competition between the aggregation and the nonlocal reaction term is highlighted: when the growth factor is stronger than the dampening effect, with the help of the nonlocal term, the model admits a classical solution which is uniformly bounded. Moreover, when the growth factor is of the same order compared to the dampening effect, the nonlocal nonlinear exponents can prevent the chemotactic collapse. Global existence of classical solutions is shown for an appropriate range of the exponents as well as convergence to the constant  equilibrium state. 
\end{abstract}


\end{frontmatter}

\section{Introduction}
Chemotaxis was originally described by the Keller-Segel model which was  introduced in \cite{KS1,KS2}. The model describes the characteristic movement of cells where there are two options: the cells can be attracted (move toward) by the increasing signal concentration or they can be repulsive by this signal. Afterwards, many mathematical models have been proposed to describe chemotaxis in the last few years. The simplest version, see for example \cite{bp07}, describes the competition between the diffusion, the reproduction and the nonlocal aggregation satisfying 
\begin{align} \label{nkpp00}
\left\{
      \begin{array}{ll}
      \ u_t=\Delta u-\chi\nabla \cdot (u^m \nabla c)+\lambda f(u), &  x \in \Omega, t>0, \\
      -\Delta c+c=u^\gamma, & x \in \Omega, t>0, \\
    \nabla u \cdot \nu=\nabla c \cdot \nu=0,  & x \in \partial \Omega, t>0, \\
     u(x,0)=u_0(x)\geq 0, &  x \in \Omega,
      \end{array}\right.
\end{align}
where $\Omega$ is a smooth bounded domain in $\mathbb{R}^n\ (n \ge 3)$, $m,\lambda,\gamma,\chi>0$ and $\nu$ is the outer unit normal vector on $\partial \Omega$. The initial data are assumed to be $u_0 \ge 0$, $u_0 \in C^{\theta}(\overline{\Omega})$ for some $\theta \in (0,1)$. The reaction term is taken to be
$$
f(u)=u^\alpha \left(1-\sigma\dashint_{\Omega} u^\beta dx\right)
$$
with $\alpha \ge 1, \beta>1,\sigma>0$. 

In the context of biological aggregation, $u(x,t)$ represents the bacteria density while $c(x,t)$ is the concentration of the chemical substance. The reaction term describes the reproduction rate of the bacteria where the resources of the environment can be consumed either locally or non-locally. The Keller-Segel system we consider describes the situation where chemicals diffuse much faster than cells \cite{JL92}, thus the original parabolic-parabolic problem can be reduced into \eqref{nkpp00}.

\subsection*{State of the art}

Since the literature for Keller-Segel models with a reaction term is vast, we will present some of the most important results. System \eqref{nkpp00} with $f(u) \equiv 0,\ m=\gamma=1$ is the classical parabolic-elliptic Keller-Segel model which expresses the random movement(brownian motion) of the cells with a bias directed by the chemoattractant concentration \cite{bp07} see also \cite{HV96,JL92,Na95} and the references therein. Moreover, for 

\begin{align*}
& u_t=\Delta u-\chi \nabla \cdot (u \nabla c),  \\
& -\Delta c=u-1,
\end{align*}

and $N=1$ blow-up never occurs \cite{Na95}, while for $N=2$ there exists a threshold depending on the initial data that separates global existence and finite time blow-up \cite{JL92}.

For $m=\gamma=1$ in \eqref{nkpp00},  the authors in \cite{TW07}  considered the reaction term  $
f(u) \le a-b u^2,\ u \ge 0,
$
which takes into account the consumption of resources around the environment. Moreover, they proved the existence of a global-in-time classical solution for either $n \le 2$ or $n \ge 3$ and $b>\frac{n-2}{2}\chi$. In addition, for all $n \ge 1, b>0$
and arbitrary initial data there exists at least one global weak solution given by $f(u) \ge -c_0(u^2+1),~\forall u>0$ with some $c_0>0.$

In \cite{Galakhov:2016bka} the authors considered:
\begin{align} \label{Gal}
\left\{
      \begin{array}{ll}
      \ u_t=\Delta u-\chi\nabla \cdot (u^m \nabla w)+\mu u\left(1- u^\alpha \right), &  x \in \Omega, t>0, \\
      -\Delta w+w=u^\gamma, & x \in \Omega, t>0, \\
    \nabla u \cdot \nu=\nabla c \cdot \nu=0,  & x \in \partial \Omega, t>0, 
      \end{array}\right.
\end{align}
and proved the existence of a unique global in time solution when
\begin{align}\label{GalCond}
\alpha &>m+\gamma-1,\quad \text{or}\\
\alpha &=m+\gamma-1,\quad \text{and}\ \mu>\frac{n\alpha-2}{2(m-1)+n\alpha}\chi.
\end{align}

Since there is a fertile area for this research, it's difficult to cover all the important results, we refer the interested readers to \cite{He:2016fz,Issa:2016uq,LS16,NakOsa11,NegTel13,SzyRoLaCha09,WL17,ZhLi15,Zh15}.

Logistic growth described by nonlocal terms has been investigated in the recent years. In \cite{NegTel13} the authors focused on the parabolic-elliptic system with a linear competitive effect:
\begin{align}\label{NT13}
\left\{
      \begin{array}{ll}
       u_t=\Delta u-\chi \nabla \cdot (u \nabla c)+u\left( a_0-a_1 u-\frac{a_2}{|\Omega|} \int_{\Omega} u dx \right), \\
     -\Delta c+ c=u+g
      \end{array}\right.
\end{align}
where $a_0,a_1>0,\chi>0,a_2 \in \mathbb{R}$ and $g$ is a uniformly bounded function. Here, as the population grows, the effect of the local term $a_1 u$ overcomes the effect of the nonlocal term $\int_{\Omega} u dx$ and therefore, the reaction term $f(u)$ behaves like $u(a_0-a_1 u)$. By using comparison methods based on upper and lower solutions, the authors proved that when $a_1>2 \chi+|a_2|$, then $\|u-\frac{a_0}{a_1+a_2}\|_{L^\infty(\Omega)} \to 0$ as time goes to infinity.

Recently, in \cite{BCL18ws} on $\mathbb{R}^n$ and in \cite{BCL18} on a bounded domain the authors studied the following chemotaxis system with a nonlocal reaction:
\begin{align} \label{nkpp00aaa}
\left\{
      \begin{array}{ll}
      \ u_t=\Delta u-\nabla \cdot (u \nabla c)+u^\alpha \left(1-\int_{\Omega} u^\beta dx\right)), &  x \in \Omega, t>0, \\
      -\Delta c+c=u, & x \in \Omega, t>0, \\
    \nabla u \cdot \nu=\nabla c \cdot \nu=0,  & x \in \partial \Omega, t>0, \\
     u(x,0)=u_0(x)\geq 0, &  x \in \Omega.
      \end{array}\right.
\end{align}
They showed that the nonlocal term $\int_{\Omega} u^\beta dx$ can actually prevent the chemotactic collapse. The result for the bounded domain case is the following:
\begin{thm*}[\cite{BCL18}]
Let $n \ge 3, \alpha \ge 1, \beta>1$ and $u_0 \in C^{\theta}(\overline{\Omega})$ for some $\theta \in (0,1)$.
If
$$
2 \le \alpha <1+2 \beta/n
$$
or
$$
\alpha<2 \mbox{~~~and~~~} \frac{n+2}{n} \left(2-\alpha\right) < 1+2\beta/n-\alpha,
$$
then the problem \eqref{nkpp00aaa} admits a unique global classical solution which is uniformly bounded. Besides, the following estimate for any $t>0$ holds true,
\begin{align}
\|u(\cdot, t)\|_{L^\infty(\Omega)} \le C(\|u_0\|_{L^1(\Omega)},\|u_0\|_{L^\infty(\Omega)}).
\end{align}
\end{thm*}

At this point we highlight the benefits of considering a nonlocal Fischer-KPP (Kolmogorov-Petrovsky-Piskunov) reaction term instead of a logistic one. The classical chemotaxis system (parabolic-elliptic) is actually a nonlocal problem therefore, in order to control the behavior of its solution it only seems natural to consider a nonlocal reaction term. As it was stated in \cite{NegTel13} that ``it seems to conjecture that the dampening effect of the nonlocal terms might lead to an even more effective homogenization". Moreover, logistic growth described by nonlocal terms has already been used in competitive systems modelling cancer cells behavior which considers the influence of the surrounding area of a cell \cite{NegTel13,SzyRoLaCha09} and it can also describe Darwinian evolution of a structured population \cite{KPP} or nuclear reaction processes \cite{HY95,WW96}. Therefore the effect of the nonlocal term on the diffusion-aggregation-reaction equation is quite attractive.

Additionally, the analytical results on chemotaxis with logistic sources mostly concentrate on local reaction terms ($a u-b u^\alpha$) where $a,b>0$:
\begin{align}\label{presum}
u_t=\Delta u-\chi \nabla u^{m} \cdot \nabla c-\chi u^m c+ \chi u^{m+\gamma}+a u-b u^\alpha.
\end{align}
Model \eqref{presum} possesses a global classical solution under the consideration that either the dampening effect is stronger than the growth factor \cite{Galakhov:2016bka,WangMuChZ14}, i.e. $\alpha>m+\gamma$, or that the dampening effect is of the same order compared to the growth factor \cite{NegTel13,TW07}, i.e. $\alpha=m+\gamma$, combining some constraints on the coefficients $b$ and $\chi$. To the best of our knowledge, when the growth factor is stronger than the dampening effect, whether the model admits a global solution is still open. Moreover, most of the known results focus on reaction terms of the type $g(u)-h(u)$ where $g(u)\sim u^a$ and $h(u)\sim u^b$ with $b>a$. The essential difference between previous problems and the present one is that the reaction term now is of the form $u^\alpha -u^\alpha \int_{\Omega} u^\beta dx$. Actually, model \eqref{nkpp00} can be rewritten as:
\begin{align*}
u_t=\Delta u-\chi\nabla u^m \cdot \nabla c+\chi u^{m+\gamma} - \chi cu^{m}+\lambda u^\alpha-\lambda  u^\alpha \int_{\Omega} u^\beta dx.
\end{align*}
Here, If the growth factor $u^{m+\gamma}$ is stronger than the dampening effect $u^\alpha$ then, the nonlocal term $\int_{\Omega} u^\beta dx$ can actually help preventing the chemotactic collapse.

\subsection*{Main results}

Our work focuses on the study of the global-in-time existence of classical solutions and their asymptotic behaviour. We prove the following results:
\begin{itemize}
\item The existence of classical solutions to \eqref{nkpp00} under appropriate conditions on the parameters. 
\item The convergence to the constant equilibrium state.
\end{itemize}

The first result on the global existence extends the result found in \cite{BCL18}. The extra nonlinearities $u^m,\ u^\gamma$ that now appear make the procedure a bit more technical, but most importantly they point out the role of these two terms on the balance between global existence and blow-up. The second one, which deals with the convergence to the constant equilibrium state, is a new result and it is the main result of this work. The  difficulty of this lies mostly on controlling the nonlinear nonlocal reaction terms (the choice of the solution pair which was used in \cite{NegTel13} does not apply here because of the nonlinearities that appear in our problem) which is handled by choosing the appropriate pair of upper-lower solutions and then using the "rectangle method" to get the convergence.

More precisely, our first result reads as follows:

\begin{thm}[Global existence]\label{thm1}
Let $n \ge 3,\ \alpha,\gamma,m \ge 1,\ \beta>1$ satisfy
\begin{align}\label{conditions}
\gamma+m \le \alpha <1+2 \beta/n
\quad\text{or}\quad
\frac{n+4}{2}-\beta<\alpha<\gamma+m,
\end{align}
then, problem \eqref{nkpp00} admits a unique global classical solution which is uniformly bounded. Moreover, the following estimate holds, then for any $t>0$,
\begin{align}
\|u(\cdot, t)\|_{L^\infty(\Omega)} \le C(\|u_0\|_{L^1(\Omega)},\|u_0\|_{L^\infty(\Omega)}).
\end{align}

\end{thm}
\begin{remark}
Our results are indeed different from the case $f(u)\equiv 0$, we have proven that the solution is uniformly bounded in $L^k(\Omega)$ for $1 \le k \le \infty$ without any restriction on the initial data. While in chemotaxis systems with $f(u) \equiv 0,$ the solution will exist globally only if one imposes a smallness condition on the initial data. 
\end{remark}

\begin{remark}
It is well known, for example in \cite{JL92}, that the solution of chemotaxis system without reaction blows up in finite time for large initial data. Comparing this result with Theorem 1 shows that an appropriate nonlocal nonlinear dampening effect could give a global in time solution without any restriction on the initial data. The condition of Theorem 1 implies $\beta>\frac{n}{2}(\gamma+m-1)$ can prevent chemotactic collapse. However, whether there exists a blow-up solution to our model under the assumption that $\beta\leq\frac{n}{2}(\gamma+m-1)$ for higher dimension is still open.
\end{remark}


Now we turn to the asymptotic behavior of the solution. We consider the spatial homogeneous version of our system (the elliptic-elliptic system) for example the integral in this case translates into: $
\sigma\dashint_\Omega \overline u^\beta\;dx=\sigma\overline u^\beta.
$

 Also, we denote
 \begin{equation}\label{equi}
 (u,c)=(\xi,\xi),\quad\text{with}\quad\xi=\sigma^{-\frac1\beta},
 \end{equation}
the constant solution to the corresponding to \eqref{nkpp00} elliptic-elliptic system.

We will use a rectangle method, see for exmple \cite{FT02}. We start by constructing a sub- and super-solution that are homogeneous in space. Therefore, we consider the following ODE system:
\begin{align}\label{odesyst}
\overline u_t
&=\chi \overline u^m(\overline u^\gamma-\underline u^\gamma)
+\lambda\sigma \overline u^\alpha (\xi^\beta-\overline u^\beta),
\nonumber\\
\underline u_t&=\chi \underline u^m(\underline u^\gamma-\overline u^\gamma)
+\lambda\sigma \underline u^\alpha (\xi^\beta-\underline u^\beta),
\end{align}
with
\begin{align}\label{odeIC}
\overline{u}(0)&>\overline{u}_0:=\max\{\max_{x\in\overline{\Omega}}u_0,\xi\},
\nonumber\\
\underline{u}(0)&>\underline{u}_0:=\inf\{\min_{x\in\overline{\Omega}}u_0,\xi\}.
\end{align}

We have the following result:

\begin{thm}[Convergence to the equilibrium]\label{thm2}
Let the assumptions of Theorem \ref{thm1} hold and $\underline u_0,\overline u_0$ as in \eqref{odeIC} to satisfy $0<\underline u_0\leq\xi\leq\overline u_0$ with 
$
\alpha+\beta\geq \gamma+m$, and $\lambda>2\chi,
$
then
$$
(u,c)\to(\xi,\xi)\quad\text{as}\quad t\to\infty.
$$
\end{thm}


This paper is organised as follows: we prove Theorems \ref{thm1},\ref{thm2} in Sections 2,3 respectively.



\section{Proof of Theorem \ref{thm1} (Global Existence)}\label{sec2}

In this Section we  prove Theorem \ref{thm1} in two steps translated into the following two Propositions but before going further we show some elementary results that we will use in this proof.

First, a Lemma which combines Gagliardo-Nirenberg and interpolation inequalities.
\begin{lemma}(\cite{BCL18}) \label{lemmainterpolation}
Let  $1 \le r<q<2^*:=\frac{2n}{n-2}$ and $\frac{q}{r}<\frac{2}{r}+1-\frac{2}{2^*}$, then for $v\in H^1(\Omega)$ and $v \in L^r(\Omega)$, it holds
\begin{align}
\|v\|_{L^q(\Omega)}^q 
\le 
C_{GNI}
\|v\|_{L^r(\Omega)}^{\gamma} 
+ C_0 
\|\nabla v\|_{L^2(\Omega)}^2
+ C_1 \|v\|_{L^2(\Omega)}^2,~~n \ge 3,
\label{inter}
\end{align}
where
$
C_{GNI}=C(n) \left(C_0^{-\frac{ \lambda q}{2-\lambda q}}  + C_1^{-\frac{ \lambda q}{2-\lambda q}}\right)
$, $C(n)$ is a constant depending on $n$ while $C_0,C_1$ are arbitrarily positive constants and
\begin{align}
\lambda=\frac{\frac{1}{r}-\frac{1}{q}}{\frac{1}{r}-\frac{1}{2^*}} \in (0,1),~~
\gamma=\frac{2(1-\lambda) q}{2-\lambda q}=\frac{2\left(1-\frac{q}{2^*}\right)}{\frac{2-q}{r}-\frac{2}{2^*}+1}.
\end{align}
\end{lemma}


Next, we pass to the proof of the global existence Theorem 2. 
It consists of two Propositions: the first one proves the local-in-time existence of a solution and the second one shows the necessary conditions on the parameters in order to get the main $L^k-$estimate, which is then used to derive the $L^\infty-$bounds after a bootstrap procedure.
First, we show the local existence of the solution to \eqref{nkpp00}.
\begin{proposition}\label{pro0}
Let $\alpha \ge 1.$ Assume $u_0\in C^{\theta}(\overline{\Omega})$ for some $\theta \in (0,1)$, then there exists a maximal existence time $T_{\max} \in (0,\infty]$ and a unique classical solution $u(x,t)$ to model \eqref{nkpp00} such that
\begin{align}
u \in C^{2+\delta,~1+\frac{\delta}{2}}\left( \overline{\Omega}\times(0,T_{\max}) \right),
\end{align}
where $\delta \in (0,1)$. Besides, if $T_{\max}<\infty,$ then
\begin{align}\label{blowup}
\displaystyle \lim_{t \to T_{\max}} \|u(t)\|_{L^\infty(\Omega)}=\infty.
\end{align}
\end{proposition}
\begin{proof}
The proof for the case $m=\gamma=1$ is in \cite{BCL18} and since it is similar to our general case we refer the interested reader to this work.	
\end{proof}

Next, we have the main ingredient of our global existence result, which is the following $L^k-$estimates:
\begin{proposition}\label{pro1}
Let $\alpha \ge 1$ and $u$ be any nonnegative classical solution of problem \eqref{nkpp00} within $0<t<T_{\max}$. If
\begin{align}\label{conditions}
\gamma+m \le \alpha <1+2 \beta/n
\quad\text{or}\quad
\frac{n+4}{2}-\beta<\alpha<\gamma+m,
\end{align}
then, the following estimate holds true that for any $0<t<T_{\max}$ and any $1 \le k < \infty$
\begin{align}
\int_{\Omega} u^k(\cdot,t) dx \le C\left( \|u_0\|_{L^1(\Omega)},\|u_0\|_{L^k(\Omega)}  \right).
\end{align}
Furthermore, the uniformly boundedness is obtained that for $0<t<T_{\max}$
\begin{align}\label{uniformbounded}
\|u(\cdot,t)\|_{L^\infty(\Omega)} \le C\left( \|u_0\|_{L^1(\Omega)},\|u_0\|_{L^\infty(\Omega)} \right).
\end{align}
\end{proposition}
\begin{proof}

\noindent\textbf{Proof of Proposition \ref{pro1}.} We begin with the boundedness of $L^k$ norm for $1<k<\infty.$

We test \eqref{nkpp00} with $k u^{k-1}(k \ge 1)$ and derive: 
\begin{align}\label{star}
& \frac{d}{dt} \int_{\Omega}  u^k dx  
+ \frac{4(k-1)}{k}\int_{\Omega} |\nabla u^{\frac{k}{2}} |^2 dx 
=
\lambda k \int_\Omega u^{k-1}f(u)dx
-\chi k\int_\Omega u^{k-1}\nabla(u^m\nabla c) dx
\end{align}
For the last term on using \eqref{nkpp00}, we get:
\begin{align}
-\chi k\int_\Omega u^{k-1}\nabla(u^m\nabla c) dx
&=
\frac{\chi k m}{k+m-1}\int_\Omega u^{k+m-1}(u^\gamma-c)dx
-\chi k\int_\Omega u^{k+m-1}(u^\gamma-c)dx,\nonumber\\
&=\chi k\left(\frac{m}{k+m-1}-1\right)\int_\Omega u^{k+m-1}(u^\gamma-c)dx
\end{align}
inserting the above into \eqref{star} we get:
\begin{align}\label{starstar}
 \frac{d}{dt} \int_{\Omega}  u^k dx  
+ \frac{4(k-1)}{k}
\int_{\Omega} |\nabla u^{\frac{k}{2}} |^2 dx 
&=
\lambda k \int_\Omega u^{k-1} f(u) dx
+C_1\int_\Omega u^{\gamma+k+m-1}dx
-C_1\int_\Omega cu^{k+m-1}dx
\nonumber\\
&\leq
\lambda k \int_\Omega u^{\alpha+k-1}\left(1-\int_\Omega u^\beta dx\right)  dx
+C_1\int_\Omega u^{\gamma+k+m-1}dx
\end{align}
where $C_1:=\chi k\left(1-\frac{m}{k+m-1}\right)\geq0$. At this point we distinguish two cases in order to determine which one is the dominant term of the right hand side.

\underline{Case 1: $\alpha\geq \gamma+m(\geq2)\Rightarrow \alpha+k-1\geq\gamma +m+k-1$.
}\\
By Young's inequality we get:
$$
\int_\Omega u^{\gamma+m+k-1}dx\leq\int_\Omega u^{\alpha+k-1} dx+C_2(\alpha,\gamma,m,k)
$$
therefore \eqref{starstar} becomes:
\begin{align}\label{s3}
 \frac{d}{dt} \int_{\Omega}  u^k dx  
+ \frac{4(k-1)}{k}
\int_{\Omega} |\nabla u^{\frac{k}{2}} |^2 dx 
+\lambda k\int_\Omega u^{\alpha+k-1}\;dx\int_\Omega u^\beta\;dx 
&\leq
C_3(k)\int_\Omega u^{\alpha+k-1}dx
+C_2'
\end{align}
and following similar ideas as in \cite{BCL15,BCL18},  for 
$$
 1+\frac{2\beta}{n}>\alpha\geq \gamma+m
$$ 
we get that for all $k\geq1$ and $t\in(0,T_{\max})$ there holds:
$$
\|u(\cdot,t)\|_k^k\leq C_4(\|u_0\|^k_k,k).
$$

\underline{Case 2: $\alpha< \gamma+m\Rightarrow \alpha+k-1<\gamma +m+k-1$.}
Similarly to the previous case \eqref{starstar} becomes:
\begin{align}\label{s4}
 \frac{d}{dt} \int_{\Omega}  u^k dx  
+ \frac{4(k-1)}{k}
\int_{\Omega} |\nabla u^{\frac{k}{2}} |^2 dx 
+\lambda k\int_\Omega u^{\alpha+k-1}\;dx\int_\Omega u^\beta\;dx 
&\leq
C_5(k)\int_\Omega u^{\gamma+m+k-1}dx
+C_6.
\end{align}

We will estimate the right hand side by applying Lemma \ref{lemmainterpolation} for 
$$~v=u^{k/2},q=\frac{2(\gamma+m+k-1)}{k},r=\frac{2k'}{k}>1,C_0=\frac{2(k-1)}{k ~C_5(k) },C_1=\frac{1}{C_5(k)}.$$

The condition $q\leq2^*$ gives  $k\geq\frac{n-2}{2}(\gamma+m-1)$ and since we started with $k>1$, this condition makes sense only if 
\begin{align}\label{kinterval}
k>\max\{1,\frac{n-2}{2}(\gamma+m-1)\}.
\end{align}

The other condition $\frac{q}{r}<\frac{2}{r}+1-\frac{2}{2^*},$ provides:
\begin{align}\label{kprime1}
k'>\frac{n(\gamma+m-1)}{n-2}
\end{align}
 and $1<r<q<2^*$ translates into
\begin{align}\label{kprime}
\frac{k}{2}<k'<\gamma+m+k+1<\frac{2^*k}{2}.
\end{align}
Then, Lemma \ref{lemmainterpolation} calculates:
\begin{align}\label{12}
\int_{\Omega} u^{\gamma+m+k+1} dx 
\le 
\frac{2(k-1)}{k ~C_5(k)} \| \nabla u^{\frac{k}{2}} \|_{2}^2
+C_{GNI} \|u\|_{k'}^{b}
+\frac{1}{C_5(k)} \|u^{k/2} \|_{2}^2
\end{align}
with
$$
b=\frac{(1-\lambda)(k+1)}{1-\frac{\lambda (k+1)}{k}},~\lambda=\frac{\frac{k}{2k'}-\frac{k}{2(k+1)}}{\frac{k}{2k'}-\frac{1}{2^*}} \in(0,1).
$$

Combining \eqref{s4} with \eqref{12} we obtain that
\begin{align}\label{13}
& \frac{d}{dt}\int_{\Omega} u^k \;dx
+\frac{2(k-1)}{k} \|\nabla u^{\frac{k}{2}} \|_{L^2(\Omega)}^2
+ \lambda k  \int_{\Omega} u^{\alpha+k-1} \;dx\int_{\Omega} u^\beta \;dx 
\nonumber \\
\le ~& 
C_5C_{GNI}\|u\|_{L^{k'}(\Omega)}^{b}+\|u\|_{L^k(\Omega)}^k+C_7.
\end{align}

At this point we aim to bound the $\|u\|_{L^{k'}}$, which appears on the right hand side, in terms of the term $\int_{\Omega} u^{\alpha+k-1} \;dx\int_{\Omega} u^\beta \;dx $, which lies on the left hand side, and afterwards the term $\|u\|_{L^k(\Omega)}^k$ should be handled by the diffusion term. 

We procceed with the first task and since,
$$
\beta<k'<k+\alpha-1,
$$
we use the following interpolation inequality
\begin{align}\label{14}
\|u\|_{k'}^{b} \le \left(\|u\|_{k+\alpha-1}^{k+\alpha-1} \|u\|_{\beta}^\beta \right)^{\frac{b \theta}{k+\alpha-1}} \|u\|_{\beta}^{b \left( 1-\theta-\frac{\theta \beta}{k+\alpha-1} \right)}
\end{align}
with
$$
\theta=\frac{\frac{1}{\beta}-\frac{1}{k'}}{\frac{1}{\beta}-\frac{1}{k+\alpha-1}} \in (0,1)
$$
to deal with \eqref{13}. Now, due to the arbitrariness of $k'$, we can specify it at this point so that:
$$
k'=\frac{k+\alpha-1+\beta}{2}
$$
such that
$$
1-\theta-\frac{\theta \beta}{k+\alpha-1}=0,
$$
and in addition if
\begin{align}\label{17}
\frac{b \theta }{k+\alpha-1}<1,
\end{align}
then we may use Young's inequality in order to infer from \eqref{14} that
\begin{align}\label{23}
C(k) \|u\|_{L^{k'}(\Omega)}^{b} & \le C(k) \left(\|u\|_{L^{k+\alpha-1}(\Omega)}^{k+\alpha-1} \|u\|_{L^\beta(\Omega)}^\beta \right)^{\frac{b \theta}{k+\alpha-1}} \nonumber\\
& \le \frac{k}{4}\|u\|_{L^{k+\alpha-1}(\Omega)}^{k+\alpha-1} \|u\|_{L^\beta(\Omega)}^\beta +C_8(k,\alpha).
\end{align}
 At this point it remains to show that  the term $\|u\|_{L^k(\Omega)}^k$ is handled by the diffusion term, so we may continue as in \cite{BCL18} Proposition 5, and prove that \eqref{17} is equivalent to the following condition:

After a few computations relation \eqref{17} is equivalent to

\begin{align}\label{19}
\frac{(2-\alpha)}{\beta}\left( \frac{k}{2}-\frac{k'}{2^*} \right)<(k+1-k')~\left( \frac{1}{2}-\frac{1}{2^*}-\frac{\alpha-1}{2 \beta} \right).
\end{align}

To simplify the following calculation we denote 
\begin{align*}
A_0=\frac{1}{2}-\frac{1}{p}-\frac{\alpha-1}{2 \beta}>0 \quad\text{and}\quad
A_1=\frac{2-\alpha}{\beta}>0.
\end{align*}
Therefore, relation \eqref{19} can be rewritten as
\begin{align}\label{20}
A_0(k+1-k')>A_1 \left( \frac{k}{2}-\frac{k'}{p}  \right)
\quad\text{or}\quad
A_0+\left(A_0- \frac{A_1}{2} \right)k+\left( \frac{A_1}{p}-A_0 \right)k'>0.
\end{align}
At this point we will clarify the ordering of the terms $A_0,\frac{A_1}{p},\frac{A_1}{2}$. If $A_0<\frac{A_1}{p}<\frac{A_1}{2}$, then \eqref{20} reads
\begin{align}\label{61}
A_0+\left( \frac{A_1}{p}-A_0 \right)k'>\left(\frac{A_1}{2}-A_0\right)k
\end{align}
and since $k'$ satisfies \eqref{kprime}, plugging $k'<\frac{2^*k}{2}$ into \eqref{61} provides
\begin{align}
\left(\frac{A_1}{2}-A_0\right)k< A_0+\left( \frac{A_1}{2^*}-A_0 \right)k'<A_0+\left( \frac{A_1}{2^*}-A_0 \right)\frac{2^*k}{2},
\end{align}
which contradicts relation \eqref{kinterval}.

Next, if $A_0>\frac{A_1}{2}>\frac{A_1}{p},$ again by using relation \eqref{kprime} we take $k'>k/2$ and relation \eqref{20} becomes
\begin{align}\label{247}
A_0+\left(A_0- \frac{A_1}{2} \right)k>\left( A_0-\frac{A_1}{2^*} \right)k'>\left( A_0-\frac{A_1}{2^*} \right) \frac{k}{2}.
\end{align}
Here we remark that if we use $k'>\frac{n(\gamma+m-1)}{n-2}$ from \eqref{kprime1} instead, then \eqref{247} reads
\begin{align}
A_0+\left(A_0- \frac{A_1}{2} \right)k>\left( A_0-\frac{A_1}{p} \right)k'>\left( A_0-\frac{A_1}{p} \right) \frac{n(\gamma+m-1)}{n-2},
\end{align}
which is equivalent to \eqref{kinterval}, thus we only need to take under consideration the case $k'>\frac{k}{2}$. 

Therefore, \eqref{17} holds true as long as $A_0-\frac{A_1}{2}>\frac{A_0}{2}-\frac{A_1}{2p},$ which translates into
\begin{align}\label{22}
\frac{(2-\alpha)}{\beta} \left(1-\frac{1}{p}\right) < \frac{1}{2}-\frac{1}{p}-\frac{\alpha-1}{2 \beta}\Leftrightarrow \alpha+\beta>\frac{n+4}{2}.
\end{align}

and again following along the lines of \cite{BCL18} we conclude that also in this case, for all $1 \le k<\infty$
\begin{align}\label{etada}
\|u(\cdot,t)\|_{L^k(\Omega)} \le C\left(k,\|u_0\|_{L^k(\Omega)}\right).
\end{align}
In a similar way, as in \cite{BCL18} step 4 of Proposition 5 we are able to get our $L^\infty-$estimates via the same iteration scheme and thus prove \eqref{uniformbounded}.

\end{proof}
We can now prove  our first result:\\
\noindent {\bf Proof of Theorem \ref{thm1}:} Now we directly make use of the $L^\infty$ estimate \eqref{uniformbounded} and the blow-up criterion \eqref{blowup} to obtain that
\begin{align}
\|u(\cdot,t)\|_{L^\infty(\Omega)} \le C(\|u_0\|_{L^1(\Omega)}, \|u_0\|_{L^\infty(\Omega)})
\end{align}
for all $t \in (0,\infty)$. $\Box$

\section{Proof of Theorem \ref{thm2} (Convergence to the equilibrium)}
 In this Section we show the convergence to the constant equilibrium for solutions to system \eqref{nkpp00}. The proof is divided into three Lemmata where in the first one we prove the ordering of the solutions to the ODE system \eqref{odesyst} with respect to the equilibrium point \eqref{equi}, in the second one the convergence of the solutions to  \eqref{odesyst} towards the constant equilibrium is proven and in the last Lemma we show the relation between these ODE solutions and our original system \eqref{nkpp00}.
 
 First we examine the order of the solution to \eqref{odesyst} with respect to the constant equilibrium.
 \begin{lemma}[ODE-ordering]\label{lem1odeOrd}
	Let $\underline u_0,\overline u_0$ defined in \eqref{odeIC} satisfy $0<\underline u_0<\xi<\overline u_0$ then, system \eqref{odesyst} has a unique solution in $(0,T)$ with $T$ defined by $\limsup_{t\to T}(|\underline u|+|\overline u|+t)=+\infty$ and
	\begin{equation}\label{odeOrder}
	0<\underline u<\xi<\overline u, \quad t\in(0,T).
	\end{equation} \end{lemma}
\begin{proof}
	Since the right hand sides of \eqref{odesyst} are continuous and locally Lipschitz functions of $\underline u,\overline u$ there exists a unique solution in $[0,T)$, also since the right hand sides are $C^1(\mathbb{R}^2)$ as functions of $\underline u,\overline u$ we get that
	$$
	(\underline u,\overline u)\in[C^2(0,T)]^2.
	$$
	Next, to show the ordering of these solutions we argue by contradiction. Assume that there exists a $t_0$ with $0<t_0<T$ such that \eqref{odeOrder} does not hold. First, we note that $\underline u(t_0)>0$ since if $\underline u(t_0)=0$ then, by  uniqueness this means that also $\underline u_0=0$ which contradicts our hypothesis. If $\xi=\underline u(t_0)<\overline u(t_0)$ then from \eqref{odesyst} at $t_0$ we have:
$$
\underline{u}'(t_0)=\chi\underline u^m(t_0)\left(\underline u^\gamma(t_0)-\overline u^\gamma(t_0)\right)+\lambda \sigma f(\underline u(t_0))<0
$$
which contradicts the assumption $\underline u(t_0)=\xi$.
With similar arguments we look into the cases $\underline u(t_0)<\overline u(t_0)=\xi$ and $\xi=\underline u(t_0)=\overline u(t_0)$ and conclude that such a $t_0$ does not exist, thus
$$
	0<\underline u<\xi<\overline u, \quad t\in(0,T).
$$
\end{proof}

Next, we prove the convergence of the solution to \eqref{odesyst} to the equilibrium $\xi$.
  \begin{lemma}[Convergence of the ODE solutions]
Let the assumptions of Theorem 1 hold and $C,\delta$ that satisfy
$$
\max\{\alpha-1,\gamma+m-1\}\leq\delta\leq\alpha+\beta-1,\quad \lambda C-2\chi>0
$$ then, $T=\infty$ and
\begin{equation}\label{odeConv}
\overline u\to \xi,\quad \underline u\to \xi.	
\end{equation}

 \end{lemma}
 \begin{proof}
Divide the first equation of \eqref{odesyst} with $\overline u$ and the second one with $\underline u$
\begin{align}\label{odesystDiv}
\frac{d}{dt}\log\overline u
&=\chi \overline u^{m-1}(\overline u^\gamma-\underline u^\gamma)
+\lambda\sigma \overline u^{\alpha-1} (\xi^\beta-\overline u^\beta),
\nonumber\\
\frac{d}{dt}\log\underline u&=\chi \underline u^{m-1}(\underline u^\gamma-\overline u^\gamma)
+\lambda\sigma \underline u^{\alpha-1} (\xi^\beta-\underline u^\beta),
\end{align}
after subtracting we get:
\begin{align}\label{ode1stbound}
\frac{d}{dt}\log \frac{\overline u}{\underline u}	
&=
\chi(\overline u^{\gamma+m-1}-\underline u^{\gamma+m-1})
-\chi(\underline u^\gamma\overline u^{m-1}-\overline u^\gamma\underline u^{m-1})
+\lambda\sigma g(\underline u,\overline u)
\nonumber\\
&\stackrel{\eqref{odeOrder}}{\leq}
2\chi(\overline u^{\gamma+m-1}-\underline u^{\gamma+m-1})
+\lambda\sigma g(\underline u,\overline u)
\end{align}
with
$$
 g(\underline u,\overline u)=\xi^\beta\overline u^{\alpha-1}-\xi^\beta \underline u^{\alpha-1}+\overline u^{\alpha+\beta-1}-\underline u^{\alpha+\beta-1}\leq C_9(\overline u^\delta-\underline u^\delta)
$$
and $\delta$ is chosen so that:
$$
\alpha-1\leq\delta\leq\alpha+\beta-1
$$
 for the bound of $g$ to hold and
$$
\gamma+m-1\leq\delta
$$
thus
$$
\max\{\alpha-1,\gamma+m-1\}\leq\delta\leq\alpha+\beta-1.
$$
Therefore, from \eqref{ode1stbound} we get
\begin{align}
\frac{d}{dt}\log\frac{\overline u}{\underline u}
\leq &(2\chi-\lambda C_9)(\overline u^\delta-\underline u^\delta)
\nonumber\\
&\stackrel{2\xi\leq\lambda C}{\leq}
-C_{10}(\overline u-\underline u)\leq-C_{10}(\log \overline u-\underline u)
\end{align}
where $C_{10}=\lambda C_9-2\chi>0$. Finally, using Gr\"onwall's Lemma and the previous Lemma \ref{lem1odeOrd} we get \eqref{odeConv}.

\end{proof}

 
 The last step before the proof of Theorem 2 is to determine the relation of this ODE solution to the solution of our original system.
\begin{lemma}[ODE-PDE ordering]
Let	$u_0\in W^{1,p}(\Omega)$ for some $p>N$ and there exists $\underline u_0>0$ such that $u_0>\underline u_0$ then, the solution to \eqref{nkpp00} for $t<T$ satisfies:
\begin{equation}
\underline u(t)\leq u(x,t)\leq \overline u(t)	.
\end{equation}

 \end{lemma}
 \begin{proof}
	We set:
	$$
	\overline U=u-\overline u,\quad \overline C=c-\overline u^\gamma
	$$
		$$
	\underline U=u-\underline u,\quad \underline C=c-\underline u^\gamma
	$$
	and 
	$$
	h(u)=\chi u^{m+\gamma}+\lambda u^\alpha(1-\sigma\dashint_\Omega u^\beta\;dx).
	$$
	Then, $\overline U$ satisfies:
	$$
	\overline{U}_t-\Delta\overline U=-m\chi u^{m-1}\nabla \overline U\cdot\nabla c+h(u)-h(\overline u) +\chi(\overline u^m\underline u^\gamma-u^mc).
	$$
	We test the above equation with the positive part of $\overline U$ i.e.  $\overline U_+$ to get:
	\begin{align}\label{testPP}
	\frac12\frac{d}{dt}\int_\Omega \overline U_+^2\;dx+\int_\Omega |\nabla\overline U_+|^2\;dx
	&=
	-m\chi\int_\Omega u^{m-1}U_+\nabla\overline U\cdot\nabla c\;dx	
	+\int_\Omega (h(u)-h(\overline u))\overline U_+\;dx
	+\chi\int_\Omega \overline U_+(\overline u^m\underline u^\gamma-u^mc)\;dx
	\nonumber\\
	&=I+II+III
	\end{align}
We treat these three integrals separately:
\begin{equation}\label{i}
I=		-m\chi\int_\Omega u^{m-1}U_+\nabla\overline U\cdot\nabla c\;dx	
\leq C_{11}\|u\|^{2m-2}_{L^\infty(\Omega_T)}\|c\|^2_{W^{1,\infty}(\Omega)}\int_\Omega \overline U^2_+\;dx+\frac12\int_\Omega |\nabla\overline U_+|^2\;dx
\end{equation}
\begin{equation}\label{ii}
II=	\int_\Omega (h(u)-h(\overline u))\overline U_+\;dx\leq C_{12}(\|u\|_\infty,h')\int_\Omega\overline U^2_+\;dx
\end{equation}
\begin{align}\label{iii}
III&=	\chi\int_\Omega \overline U_+(\overline u^m\underline u^\gamma-u^mc)\;dx=
\chi\int_\Omega \overline U_+[(\overline u^m-u^m)\underline u^\gamma-u^m(c-\underline u^\gamma)]\;dx\nonumber\\
&\leq-\chi\int_\Omega u^m\underline C\overline U_+\;dx
\leq\chi\|u\|^m_{L^\infty(\Omega)}\int_\Omega\underline C_-\overline U_+\;dx
\leq\frac12\chi\|u\|^m_{L^\infty(\Omega)}\int_\Omega\underline{C}_-^2+\overline U_+^2\;dx
\end{align}
Next testing
$$
\Delta \underline C+\underline C=u^\gamma-\underline u^\gamma
$$
with $-\underline C_-$ we get:
\begin{equation}\label{iv}
\int_\Omega|\nabla \underline C_-|^2\;dx+\frac12\int_\Omega|\underline C_-|^2\;dx
\leq\int_\Omega (u^\gamma-\underline u^\gamma)^2_-\;dx
\stackrel{m,\gamma\geq1}{\leq}C_{13}(\gamma,\|u\|_{L^\infty(\Omega)})\int_\omega\underline U^2_-\;dx
\end{equation}

Plugging \eqref{i}, \eqref{ii}, \eqref{iii}, \eqref{iv} into \eqref{testPP} we get:
\begin{equation}\label{star1}
\frac12\frac{d}{dt}\int_\Omega \overline U^2_+\;dx
+\frac12\int_\Omega |\nabla \overline U_+^2\;dx	
\leq C_{14}(\|u\|_{L^\infty(\Omega)},h')\int_\Omega (\overline U^2_++\underline U^2_-)\;dx
\end{equation}
similarly,
\begin{equation}\label{star2}
\frac12\frac{d}{dt}\int_\Omega \underline U^2_-\;dx
+\frac12\int_\Omega |\nabla \underline U_-|^2\;dx	
\leq C_{15}(\|u\|_{L^\infty(\Omega)},h')\int_\Omega (\overline U^2_++\underline U^2_-)\;dx
\end{equation}
adding \eqref{star1} to \eqref{star2} and using Gr\"onwall's Lemma we get
$$
\overline U_+=\underline U_-=0,\quad \forall t\leq \widehat T< T
$$
the result follows after taking the limit $\widehat T\to T$.
\end{proof}
Next, we turn to the proof of our second Theorem:
\begin{proof}[Proof of Theorem \ref{thm2}]
Since $T=\infty$ and 
\begin{align}
\label{est1}\underline u&\leq u\leq \overline u,\\
\label{est2}\underline u^m&\leq u\leq \overline u^m,		
\end{align}
we get that $u,c\in L^\infty(\Omega)$, actually $(u,c)$ are uniformly bounded in $(0,\infty)$. Next, we notice that
\begin{align*}
	\|u-\xi\|_{L^\infty(\Omega)}&\leq |\overline u-\xi|+\|\overline u-u\|_{L^\infty(\Omega)}\stackrel{\eqref{est1}}{\leq}|\overline u-\underline u|\to0
	\\
	\|c-\xi\|_{L^\infty(\Omega)}&\leq |\overline u^m-\xi|+\|\overline u^m-c\|_{L^\infty(\Omega)}\stackrel{\eqref{est2}}{\leq}|\overline u^m-\underline u^m|\to0
\end{align*}
which completes the proof.
\end{proof}

\section*{Acknownledgements}
The author was supported by the project: SFB Lipid Hydrolysis.

\end{document}